\documentclass[leqno,12pt]{article} 
\usepackage{amsmath, amssymb}
\usepackage{amsthm} 
\theoremstyle{plain} 
\newtheorem{theorem}{\indent\sc Theorem}[section]

\newtheorem{corollary}[theorem]{\indent\sc Corollary}
\newtheorem{proposition}[theorem]{\indent\sc Proposition}

\theoremstyle{definition} 

\newtheorem{remark}[theorem]{\indent\sc Remark}
\newtheorem{example}[theorem]{\indent\sc Example}

\newcommand\on{\operatorname}
\renewcommand\div{\on{div}}
\newcommand\grad{\on{grad}}
\newcommand\Hess{\on{Hess}}
\newcommand\Ric{\on{Ric}}
\newcommand\scal{\on{scal}}

\title{Geometric solitons in a
$D$-homothetically deformed Kenmotsu manifold}
\author{Adara M. Blaga}
\date{}
\begin{document}

\maketitle

\markboth{{\small\it {\hspace{2cm} Geometric solitons in a
$D$-homothetically deformed Kenmotsu manifold}}}{\small\it{Geometric solitons in a
$D$-homothetically deformed Kenmotsu manifold
\hspace{2cm}}}


\begin{abstract}
We consider almost Riemann and almost Ricci solitons in a $D$-homothetically deformed Kenmotsu manifold having as potential vector field a gradient vector field, a solenoidal vector field or the Reeb vector field of the deformed structure, and explicitly obtain the Ricci and scalar curvatures for some cases. We also provide a lower bound for the Ricci curvature of the initial Kenmotsu manifold when the deformed manifold admits a gradient almost Riemann or almost Ricci soliton.
\end{abstract}

\bigskip

\section{Preliminaries}

Riemann and Ricci solitons are stationary solutions of the Riemann and Ricci flow, respectively \cite{ham}. In recent years, various geometric properties of these types of solitons have been studied from different points of view, by proving the existence and obstructions on curvature. More or less natural generalizations of these geometric solitons ($*$-Ricci solitons, $\eta$-Ricci solitons, generalized Ricci solitons, Ricci-Yamabe solitons, Ricci-Bourguignon solitons, Newton-Ricci solitons, Cotton solitons, Schouten solitons etc.) have been introduced on (pseudo)-Riemannian manifolds carrying an additional structure such as almost product, almost complex, almost contact, statistical structure etc. and treated the cases when the potential vector field of the soliton is pointwise collinear or orthogonal to the Reeb vector field of the structure, or when it is a conformal vector field.


In the present paper, we shall consider almost Riemann and almost Ricci solitons in a $D$-homothetically deformed Kenmotsu ma\-ni\-fold. $D$-homothetic deformations were introduced by Tanno \cite{ta} to obtain results on the second Betti numbers and harmonic forms. In almost contact metric geometry, a $D$-homothetic deformation preserves the property of a structure $(\phi,\xi,\eta,g)$ of being $K$-contact or Sasakian. The name derives form the fact that, the metrics
restricted to the contact distribution $D:=\ker (\eta)$, are homothetic.

After a short preliminary section, in Section 2 we give a brief description of the $D$-homothetically deformed Kenmotsu structure, deducing also the expressions of the Hessian, gradient, divergence and Laplace operators w.r.t. the deformed metric. We also provide the relation between the Hilbert-Schmidt norm w.r.t. the two metrics.
In Sections 3, 4 and 5 we consider almost Riemann and almost Ricci solitons in the deformed manifold having as potential vector field the deformed Reeb vector field, a gradient vector field, i.e. equal to the gradient of a smooth function, or a solenoidal vector field, i.e. divergence-free, and in the gradient Riemann and gradient Ricci case, we estimate the norm of the Ricci tensor field and for compact manifolds, we make some remarks on the volume of $M$. Notice that gradient vector fields and solenoidal vector fields often appear while modeling physical phenomena. Important examples of solenoidal vector fields are the magnetic vector fields, whose trajectories are magnetic curves. They describe the magnetic influence of electric charges in magnetic media, having practical applications. It's obvious that the harmonic functions provide examples of solenoidal vector fields of gradient type, namely, if $\Delta(f)=0$, then $\grad(f)$ is a solenoidal vector field.


\section{Deformed almost contact metric structures}

\textit{An almost contact metric manifold} is an odd dimensional Riemannian manifold $(M,g)$ with a $(1,1)$-tensor field $\phi$, a vector field $\xi$ (called the Reeb vector field) and a $1$-form $\eta$ satisfying \cite{sas}:
$$\phi^2=-I+\eta\otimes \xi, \ \ \eta(\xi)=1,$$
$$\eta=i_{\xi}g, \ \ g(\phi \cdot,\phi \cdot)=g-\eta\otimes \eta.$$

From the definition one gets:
$$\phi \xi=0, \ \ \eta\circ \phi=0, \ \ g(\phi \cdot, \cdot)=-g(\cdot, \phi \cdot).$$

An almost contact metric manifold $(M,\phi,\xi,\eta,g)$ is called \textit{Kenmotsu manifold} if for any $X$, $Y\in \mathfrak{X}(M)$,
the Levi-Civita connection $\nabla$ of $g$ satisfies
\begin{equation} \label{e2}
(\nabla _{X}\phi) Y=g(\phi X,Y) \xi
-\eta (Y)\phi X.
\end{equation}


Let $(M,\phi,\xi,\eta,g)$ be a $(2n+1)$-dimensional almost contact metric ma\-ni\-fold and denote by $D:=\ker(\eta)$ the contact distribution. For $a$ a positive constant, we define the $D$-homothetically deformation \cite{blair}:
$$\overline{\phi}=\phi, \ \ \overline{\xi}=\frac{1}{a}\xi, \ \ \overline{\eta}=a\eta, \ \ \overline{g}=ag+a(a-1)\eta\otimes \eta.$$
Then $(M,\overline{\phi},\overline{\xi},\overline{\eta},\overline{g})$ is also an almost contact metric manifold.

\bigskip

In the rest of the paper, we will consider $D$-homothetic deformations of a Kenmotsu manifold and we shall characterize some particular types of solitons in the deformed manifold, with a special view towards curvature.

For a $(2n+1)$-dimensional Kenmotsu manifold $(M,\phi,\xi,\eta,g)$, if we set $Y=\xi$ in (\ref{e2}), by a direct computation, we get:
\begin{equation}\label{kk}
\nabla\xi=I-\eta\otimes \xi,
\end{equation}
\begin{equation}\label{ll}
\pounds _{\xi}g=2(g-\eta\otimes \eta),
\end{equation}
$$\div(\xi)=2n.$$

Then the Levi-Civita connection $\overline{\nabla}$ of the deformed structure is
$$\overline{\nabla}=\nabla+\frac{a-1}{a}(g-\eta\otimes \eta)\otimes \xi,$$
whose Riemann curvature tensors of type $(1,3)$ and $(0,4)$, Ricci curvature tensor and scalar curvature are respectively given by
$$\overline{R}(X,Y)Z=R(X,Y)Z+\frac{a-1}{a}[g(\phi Y,\phi Z)X-g(\phi X,\phi Z)Y],$$
\begin{equation}\label{33}
\overline{R}(X,Y,Z,W)=aR(X,Y,Z,W)+(a-1)\{\eta(Z)[\eta(X)g(Y,W)-\eta(Y)g(X,W)]-\end{equation}$$-g(X,Z)[g(Y,W)-\eta(Y)\eta(W)]+g(Y,Z)[g(X,W)-\eta(X)\eta(W)]\},
$$
\begin{equation}\label{333}
\overline{\Ric}=\Ric+\frac{2n(a-1)}{a}(g-\eta\otimes \eta),
\end{equation}
\begin{equation}\label{s}
\overline{\scal}=\frac{1}{a}\scal+\frac{2n(2n+1)(a-1)}{a^2}.
\end{equation}

Moreover, by a direct computation, we get:
$$
(\overline{\nabla}_X\overline{\phi})Y=(\nabla_X\phi)Y-\frac{a-1}{a}g(\phi X,Y)\xi,
$$
$$\overline{\nabla}\overline{\xi}=\frac{1}{a}\nabla\xi,
$$
\begin{equation}\label{li}
\pounds _{\overline{\xi}}\overline{g}=\pounds _{\xi}g,
\end{equation}
$$
\overline{\div}(\overline{\xi})=\frac{1}{a}\div(\xi)
$$
and for any $f\in C^{\infty}(M)$, the Hessian, the gradient, the divergence and the Laplace operators w.r.t. $\overline{g}$ satisfy:
$$
\overline{{\Hess}}(f)={\Hess}(f)-\frac{a-1}{a}\xi(f)(\flat_g\circ \nabla \xi),
$$
$$
\overline{\grad}(f)=\frac{1}{a}\grad(f)-\frac{a-1}{a^2}\xi(f)\xi,
$$
$$
\overline{{\div}}=\div,
$$
$$
\overline{{\Delta}}(f)=\frac{1}{a}\Delta(f)-\frac{a-1}{a^2}\xi(f)\div(\xi)-\frac{a-1}{a^2}\xi(\xi(f)).
$$

Therefore

\begin{proposition}
In a $(2n+1)$-dimensional $D$-homothetically deformed Kenmotsu manifold, we have:
$$
(\overline{\nabla}_X\overline{\phi})Y=\frac{1}{a}g(\phi X,Y)\xi-\eta(Y)\phi X,
$$
$$
\overline{\nabla}\overline{\xi}=\frac{1}{a}(I-\eta\otimes \xi),
$$
\begin{equation}\label{l}
\pounds _{\overline{\xi}}\overline{g}=2(g-\eta\otimes \eta),
\end{equation}
\begin{equation}\label{di}
\overline{\div}(\overline{\xi})=\frac{2n}{a}.
\end{equation}

Moreover, for any $f\in C^{\infty}(M)$, the Hessian, the gradient, the divergence and the Laplace operators w.r.t. $\overline{g}$ satisfy:
\begin{equation}\label{he}
\overline{{\Hess}}(f)={\Hess}(f)-\frac{a-1}{a}\xi(f)(g-\eta\otimes \eta),\end{equation}
$$
\overline{\grad}(f)=\frac{1}{a}\grad(f)-\frac{a-1}{a^2}\xi(f)\xi,
$$
$$
\overline{{\div}}=\div,
$$
\begin{equation}\label{la}
\overline{{\Delta}}(f)=\frac{1}{a}\Delta(f)-\frac{2n(a-1)}{a^2}\xi(f)-\frac{a-1}{a^2}\xi(\xi(f)).
\end{equation}
\end{proposition}

\begin{remark}
From (\ref{la}), we notice that a $\Delta$-harmonic function $f$ is also $\overline{{\Delta}}$-harmonic if and only if $\Hess(f)(\xi,\xi)=-2n\eta(\grad(f))$.
\end{remark}

Let $\overline{E}_i=\frac{1}{\sqrt{a}}E_i$, for $i\in \{1,\dots, 2n\}$, and $\overline{\xi}=\frac{1}{a}\xi$ define two orthonormal basis w.r.t. $\overline{g}$ and $g$, respectively. Then for any symmetric $(0,2)$-tensor fields $T_1$ and $T_2$, computing the inner product w.r.t. $\overline{g}$ and $g$, we obtain:
\begin{equation}\label{pa}
\langle T_1,T_2\rangle_{\overline{g}}=\frac{1}{a^2}\langle T_1,T_2\rangle_{g}-\frac{a^2-1}{a^4}T_1(\xi,\xi)\cdot T_2(\xi,\xi),
\end{equation}
hence the Hilbert-Schmidt norms of a symmetric $(0,2)$-tensor field $T$ w.r.t. $\overline{g}$ and $g$ satisfy
\begin{equation}\label{par}
|T|_{\overline{g}}^2=\frac{1}{a^2}|T|_{g}^2-\frac{a^2-1}{a^4}[T(\xi,\xi)]^2
\end{equation}
and taking into account that
$$\langle g,g \rangle_g=2n+1, \ \ \langle g,\Ric \rangle_g=\scal, \ \ \langle g,\Hess(f) \rangle_g=\Delta(f), \ \ \langle g,\eta\otimes \eta \rangle_g= 1,$$
$$\langle \Ric, \Hess(f) \rangle_g=\sum_{i=1}^{2n}\Ric(\nabla_{E_i}\grad(f),E_i)+\Ric(\nabla_{\xi}\grad(f),\xi), $$
$$\langle \Ric, \eta\otimes \eta \rangle_g=-2n, \ \ \langle \Hess(f), \eta\otimes \eta \rangle_g=\xi(\xi(f)), \ \ \langle \eta\otimes \eta, \eta\otimes \eta \rangle_g=1,$$
by a direct computation, we obtain:
\begin{proposition}
In a $(2n+1)$-dimensional $D$-homothetically deformed Kenmotsu manifold, we have:
$$
\langle g,\Ric \rangle_{\overline{g}}=\frac{1}{a^2} \scal+\frac{2n(a^2-1)}{a^4},
$$
$$
\langle g,\Hess(f) \rangle_{\overline{g}}=\frac{1}{a^2} \Delta(f)-\frac{a^2-1}{a^4}\xi(\xi(f)),
$$
$$
\langle g,\eta\otimes \eta \rangle_{\overline{g}}=\frac{1}{a^4},
$$
$$
\langle \Ric,\Hess(f) \rangle_{\overline{g}}=\frac{1}{a^2} \langle \Ric,\Hess(f) \rangle_{g}+\frac{2n(a^2-1)}{a^4}\xi(\xi(f)),
$$
$$
\langle \Ric,\eta\otimes \eta \rangle_{\overline{g}}=-\frac{2n}{a^4},
$$
$$
\langle \Hess(f),\eta\otimes \eta \rangle_{\overline{g}}=\frac{1}{a^4}\xi(\xi(f)).
$$

Also, the Hilbert-Schmidt norms w.r.t. $\overline{g}$ and $g$ satisfy:
$$
|g|_{\overline{g}}^2=\frac{2na^2+1}{a^4},
$$
$$
|\Ric|_{\overline{g}}^2=\frac{1}{a^2}|\Ric|_{g}^2-\frac{4n^2(a^2-1)}{a^4},
$$
$$
|\Hess(f)|_{\overline{g}}^2=\frac{1}{a^2}|\Hess(f)|_{g}^2-\frac{a^2-1}{a^4}(\xi(\xi(f)))^2,
$$
$$
|\eta\otimes \eta|_{\overline{g}}^2=\frac{1}{a^4}.
$$
\end{proposition}

\begin{remark}
From the previous proposition, we deduce that in a $(2n+1)$-dimensional $D$-homothetically deformed Kenmotsu manifold, the Ricci curvature satisfies:
$$|\Ric|_{g}^2\geq \frac{4n^2(a^2-1)}{a^2}.$$

In particular, $a$ can not take any positive real value, precisely, if $|\Ric|_{g}^2<4n^2$, then $a\in \left(0, \frac{4n^2}{4n^2-|\Ric|_{g}^2}\right)$.
\end{remark}


In the next sections we shall characterize almost Riemann and almost Ricci solitons in a $D$-homothetically deformed Kenmotsu ma\-ni\-fold when the potential vector field of the soliton is the deformed Reeb vector field $\overline{\xi}$, a solenoidal or a gradient vector field.

\section{Almost Riemann solitons}

On a $(2n+1)$-dimensional smooth manifold $M$, a Riemannian metric $g$ and a non-vanishing vector field $V$ is said to define
\textit{a Riemann soliton} \cite{hi} if there exists a real constant $\lambda$ such that
\begin{equation}\label{2}
\frac{1}{2}\pounds _{V}g\odot g+R=\lambda G,
\end{equation}
where $G:=\frac{1}{2}g\odot g$, $\pounds _{V}$ denotes the Lie derivative operator in the direction of the vector field $V$ and $R$
is the Riemann curvature tensor.
If $\lambda$ is a smooth function on $M$, we call $(V,\lambda)$ an \textit{almost Riemann soliton}. Moreover, if $V$ is a gradient vector field, we call $(V,\lambda)$ a \textit{gradient almost Riemann soliton}. A Riemann soliton defined by $(V,\lambda)$ is said to be \textit{shrinking}, \textit{steady} or \textit{expanding} according as $\lambda$ is positive, zero or negative, respectively.

It was proved \cite{tripathi} that the Riemann and Ricci tensor fields of a $(2n+1)$-dimensional Kenmotsu manifold satisfy:
$$
R(X,Y)\xi=\eta(X)Y-\eta(Y)X
$$
and
$$\Ric(\xi,\xi)=-2n.$$

Since the Kulkarni-Nomizu product for two $(0,2)$-tensor fields $T_1$ and $T_2$ is defined by:
$$(T_1 \odot T_2)(X,Y,Z,W):=T_1(X,W)T_2(Y,Z)+T_1(Y,Z)T_2(X,W)-$$$$-T_1(X,Z)T_2(Y,W)-T_1(Y,W)T_2(X,Z),$$
for any $X,Y,Z,W\in \mathfrak{X} (M)$, then the Riemann soliton equation (\ref{2}) is explicitly expressed as
\begin{equation}\label{3}
2R(X,Y,Z,W)+[g(X,W)(\pounds _{V}g)(Y,Z)+g(Y,Z)(\pounds _{V}g)(X,W)-
\end{equation}
$$-g(X,Z)(\pounds _{V}g)(Y,W)-g(Y,W)(\pounds _{V}g)(X,Z)]=$$$$=2\lambda [g(X,W)g(Y,Z)-g(X,Z)g(Y,W)],$$
which by contraction over $X$ and $W$, gives
\begin{equation}\label{4}
\frac{1}{2}\pounds _{V}g+\frac{1}{2n-1}\Ric=\frac{2n\lambda-\div(V)}{2n-1}g
\end{equation}
and further
\begin{equation}\label{9}
\scal=2n[(2n+1)\lambda-2\div(V)].
\end{equation}

From (\ref{33}) and (\ref{3}), we obtain:

\begin{proposition}
If $(\overline{\xi},\overline{\lambda})$ defines an almost Riemann soliton in a $D$-homo\-the\-ti\-cally deformed Kenmotsu manifold $(M,\overline{\phi},\overline{\xi},\overline{\eta},\overline{g})$, then:
$$R(X,Y,Z,W)=-2[g(X,W)g(Y,Z)-g(X,Z)g(Y,W)]+$$$$
+g(X,W)\eta(Y)\eta(Z)-g(X,Z)\eta(Y)\eta(W)+$$$$+g(Y,Z)\eta(X)\eta(W)-g(Y,W)\eta(X)\eta(Z),$$
for any $X,Y,Z,W\in \mathfrak{X}(M)$.
\end{proposition}

\begin{theorem}\label{p2}
If $(\overline{\xi},\overline{\lambda})$ defines an almost Riemann soliton in a $(2n+1)$-dimensional $D$-homothetically deformed Kenmotsu manifold $(M,\overline{\phi},\overline{\xi},\overline{\eta},\overline{g})$, then:
$$
\overline{\lambda}=\frac{a-1}{a^2},
$$
\begin{equation}\label{r1}
\Ric=-(4n-1)g+(2n-1)\eta\otimes \eta,
\end{equation}
\begin{equation}\label{s1}
\scal=-8n^2.
\end{equation}
\end{theorem}
\begin{proof}
Equating (\ref{333}) and (\ref{4}) and using (\ref{l}) and (\ref{di}), we get:
\begin{equation}\label{B}
\Ric=\left[2na\overline{\lambda}-(4n-1)-2n\frac{a-1}{a}\right]g+\end{equation}$$+\left[2na(a-1)\overline{\lambda}+(4n-1-2na)+2n\frac{a-1}{a}\right]\eta\otimes \eta.$$

Also, by equating (\ref{s}) and (\ref{4}) and using (\ref{di}), we obtain:
\begin{equation}\label{C}
\scal=2n(2n+1)a\overline{\lambda}-8n^2-2n(2n+1)\frac{a-1}{a}.
\end{equation}

Now, tracing (\ref{B}) and considering (\ref{C}), we get:
$$\overline{\lambda}=\frac{1}{a}-\frac{1}{a^2}$$
which, by replacing it in (\ref{B}) and (\ref{C}), gives (\ref{r1}) and (\ref{s1}).
\end{proof}

\begin{remark}
Under the hypotheses of Theorem \ref{p2}, we get $$|\Ric|^2=2n(16n^2-6n+1).$$
\end{remark}

\begin{proposition}
If $(\overline{\xi},\overline{\lambda})$ defines an almost Riemann soliton in $(M,\overline{\phi},\overline{\xi},\overline{\eta},\overline{g})$, then $(\xi,\lambda)$ defines an almost Riemann soliton in $(M,\phi,\xi,\eta,g)$ if and only if $\lambda=0$.
\end{proposition}

\begin{theorem}\label{hh}
If $(V,\overline{\lambda})$ defines an almost Riemann soliton of solenoidal type in a $(2n+1)$-dimensional $D$-homothetically deformed Kenmotsu manifold $(M,\overline{\phi},\overline{\xi},\overline{\eta},\overline{g})$, then:
$$
\overline{\lambda}=\frac{2n-1}{2n}\xi(\eta(V))-\frac{1}{a^2},
$$
\begin{equation}\label{r2}
\Ric=[(2n-1)a\xi(\eta(V))-2n]g(X,Y)+(2n-1)a(a-1)\xi(\eta(V))\eta(X)\eta(Y)-\end{equation}
$$-\frac{(2n-1)a}{2}[g(\nabla_XV,Y)+g(\nabla_YV,X)]-$$$$-
\frac{(2n-1)a(a-1)}{2}\{\eta(X)[\eta(\nabla_YV)+g(Y,V)]+\eta(Y)[\eta(\nabla_XV)+g(X,V)]-$$$$-2\eta(X)\eta(Y)\eta(V)\},
$$
\begin{equation}\label{s2}
\scal=(2n+1)[(2n-1)a\xi(\eta(V))-2n].
\end{equation}
\end{theorem}
\begin{proof}
Equating (\ref{333}) and (\ref{4}) and using $\div(V)=0$, we get:
\begin{equation}\label{B1}
\Ric(X,Y)=\left(2na\overline{\lambda}-2n\frac{a-1}{a}\right)g(X,Y)+\left[2na(a-1)\overline{\lambda}+2n\frac{a-1}{a}\right]\eta(X)\eta(Y)-\end{equation}
$$-(2n-1)\frac{a(a-1)}{2}\{\eta(X)[\eta(\nabla_YV)+g(Y,V)]+
\eta(Y)[\eta(\nabla_XV)+g(X,V)]-$$
$$-2\eta(X)\eta(Y)\eta(V)\}-(2n-1)\frac{a}{2}[g(\nabla_XV,Y)+g(\nabla_YV,X)].$$

Also, by equating (\ref{s}) and (\ref{9}), we obtain:
\begin{equation}\label{C1}
\scal=2n(2n+1)a\overline{\lambda}-2n(2n+1)\frac{a-1}{a}.
\end{equation}

Now, tracing (\ref{B1}) and considering (\ref{C1}), we get:
$$\overline{\lambda}=\frac{2n-1}{2n}\xi(\eta(V))-\frac{1}{a^2}$$
which, by replacing it in (\ref{B1}) and (\ref{C1}), gives (\ref{r2}) and (\ref{s2}).
\end{proof}

\begin{corollary}
Under the hypotheses of Theorem \ref{hh}, if $V$ is $g$-orthogonal to $\xi$, then $(M,g)$ is of negative constant scalar curvature and $(V,\overline{\lambda})$ is an expanding Riemann soliton.
\end{corollary}

\begin{theorem}\label{k}
If $(V=\overline{\grad}(f),\overline{\lambda})$ defines a gradient almost Riemann soliton in a $(2n+1)$-dimensional $D$-homothetically deformed Kenmotsu manifold $(M,\overline{\phi},\overline{\xi},\overline{\eta},\overline{g})$, then:
\begin{equation}\label{lam}
\overline{\lambda}=\frac{1}{2na}{\Delta}(f)-\frac{a-1}{a^2}\eta(\grad(f))+\frac{2n-a}{2na^2}\Hess(f)(\xi,\xi)-\frac{1}{a^2}.
\end{equation}

\end{theorem}
\begin{proof}
If $V=\overline{\grad}(f)$, then $\frac{1}{2}\pounds _{V}\overline{g}=\overline{\Hess}(f)$ and equation (\ref{4}) becomes
\begin{equation}\label{rr}
\overline{\Hess}(f)+\frac{1}{2n-1}\overline{\Ric}=\frac{2n\overline{\lambda}-\overline{\Delta}(f)}{2n-1}\overline{g}.
\end{equation}

Using (\ref{333}), (\ref{he}) and (\ref{la}) in (\ref{rr}) we get
\begin{equation}\label{pp}
(2n-1)\Hess(f)+\Ric+\frac{a-1}{a}[2n-(2n-1)\xi(f)](g-\eta\otimes \eta)=\end{equation}$$=\left[2n\overline{\lambda}-\frac{1}{a}\Delta(f)+\frac{2n(a-1)}{a^2}\xi(f)+\frac{a-1}{a^2}\xi(\xi(f))\right][ag+a(a-1)\eta\otimes \eta]
$$
and tracing (\ref{pp}) w.r.t. $g$, we obtain
\begin{equation}\label{oo}
\scal=2na(2n+a)\overline{\lambda}-(4n+a-1)\Delta(f)+\frac{2n(a-1)(4n+a-1)}{a}\xi(f)+\end{equation}$$+\frac{(2n+a)(a-1)}{a}\xi(\xi(f))-\frac{4n^2(a-1)}{a}.
$$

Now tracing (\ref{rr}) w.r.t. $\overline{g}$ we get
$$
4n \overline{\Delta}(f)+\overline{\scal}=2n(2n+1)\overline{\lambda}
$$
and using (\ref{s}) and (\ref{la}), we obtain
\begin{equation}\label{ooo}
\scal=2na(2n+1)\overline{\lambda}-4n\Delta(f)+\frac{8n^2(a-1)}{a}\xi(f)+\end{equation}$$+\frac{4n(a-1)}{a}\xi(\xi(f))-\frac{2n(2n+1)(a-1)}{a}.
$$

Now, equating (\ref{oo}) and (\ref{ooo}), taking into account that $\xi(\xi(f))=\Hess(f)(\xi,\xi)$, we get (\ref{lam}).

\end{proof}

\begin{example}
Consider the $3$-dimensional Kenmotsu manifold $(M, \phi,\xi,\eta,g)$, where $M=\{(x,y,z)\in\mathbb{R}^3, z>1\}$, with $(x,y,z)$ the standard coordinates in $\mathbb{R}^3$, and
$$\phi:= dx\otimes \frac{\partial}{\partial y}-dy\otimes \frac{\partial}{\partial x}, \ \ \xi:=\frac{\partial}{\partial z}, \ \ \eta:= dz, \ \ g:=e^{2z}(dx\otimes dx+dy\otimes dy)+dz\otimes dz.$$

Then the pair $(V=e^z\frac{\partial}{\partial z}, \lambda=2e^z-1)$ defines a shrinking gradient almost Riemann soliton \cite{bl}, with
$V=\grad(f)$, for $f(x,y,z)=e^z$.

If $(\overline{V}=\overline{\grad}(f),\overline{\lambda})$ defines a gradient almost Riemann soliton in the $D$-homothetically deformed Kenmotsu manifold $(M,\overline{\phi},\overline{\xi},\overline{\eta},\overline{g})$, then
$$\overline{V}= \frac{1}{a^2}V, \ \ \overline{\lambda}=\frac{1}{a^2}\lambda.$$
\end{example}

\begin{corollary}
Under the hypotheses of Theorem \ref{k}, if $V$ is $\overline{g}$-orthogonal to $\xi$, we get
\begin{equation}\label{as}
\overline{\lambda}=\frac{1}{2na}{\Delta}(f)-\frac{1}{a^2},
\end{equation}
\begin{equation}\label{ds}
\scal=-(2n-1)\Delta(f)-2n(2n+1).
\end{equation}
\end{corollary}

Considering (\ref{rr}) and computing the Hilbert-Schmidt norms, if $(V=\overline{\grad}(f),\overline{\lambda})$ defines a gradient almost Riemann soliton in a $(2n+1)$-dimensional $D$-homothetically deformed Kenmotsu manifold $(M,\overline{\phi},\overline{\xi},\overline{\eta},\overline{g})$, then
$$
|\overline{\Hess}(f)|_{\overline{g}}^2=\frac{1}{(2n-1)^2}[|\overline{\Ric}|_{\overline{g}}^2-4n^2(2n+1)\overline{\lambda}^2+16n^2\overline{\Delta}(f)\overline{\lambda}-(6n-1)(\overline{\Delta}(f))^2]
$$
and imposing the existence condition on $\overline{\lambda}$, we get:

\begin{theorem}
If $(V=\overline{\grad}(f),\overline{\lambda})$ defines a gradient almost Riemann soliton in a $(2n+1)$-dimensional $D$-homothetically deformed Kenmotsu manifold $(M,\overline{\phi},\overline{\xi},\overline{\eta},\overline{g})$, then
\begin{equation}\label{n}
|\overline{\Ric}|_{\overline{g}}^2\geq (2n-1)^2\left[|\overline{\Hess}(f)|_{\overline{g}}^2-\frac{1}{2n+1}(\overline{\Delta}(f))^2\right].
\end{equation}
\end{theorem}

\begin{corollary}
If $(V=\overline{\grad}(f),\overline{\lambda})$ defines a solenoidal gradient almost Riemann soliton in a $(2n+1)$-dimensional $D$-homothetically deformed Kenmotsu manifold $(M,\overline{\phi},\overline{\xi},\overline{\eta},\overline{g})$, then
\begin{equation}
|\overline{\Ric}|_{\overline{g}}^2\geq (2n-1)^2|\overline{\Hess}(f)|_{\overline{g}}^2.
\end{equation}
\end{corollary}

Replacing now $\overline{\Ric}$, $\overline{\Hess}(f)$ and $\overline{\Delta}$ from (\ref{333}), (\ref{he}) and (\ref{la}) in (\ref{n}), we obtain:
\begin{proposition}\label{hhah}
If $(V=\overline{\grad}(f),\overline{\lambda})$ defines a gradient almost Riemann soliton in a $(2n+1)$-dimensional $D$-homothetically deformed Kenmotsu manifold $(M,\overline{\phi},\overline{\xi},\overline{\eta},\overline{g})$, then
$$
|\Ric|_{g}^2\geq (2n-1)^2|{\Hess}(f)|_{g}^2-4n\frac{a-1}{a}\scal-4n^2(2n+1)\left(\frac{a-1}{a}\right)^2-$$
$$-\frac{(2n-1)^2}{2n+1}({\Delta}(f))^2-\frac{2(2n-1)^2}{2n+1}\frac{a-1}{a}[\xi(f)-\xi(\xi(f))]\Delta(f)+$$
$$+\frac{2n(2n-1)^2}{2n+1}\left(\frac{a-1}{a}\right)^2(\xi(f))^2-\frac{2(2n-1)^2(n+na+a)(a-1)}{(2n+1)a^2}(\xi(\xi(f)))^2+$$
$$+\frac{2(2n-1)^2(2n+a)(a-1)}{(2n+1)a^2}\xi(f)\cdot \xi(\xi(f)).$$
\end{proposition}

\begin{corollary}
Under the hypotheses of Proposition \ref{hhah}, if $V$ is $\overline{g}$-orthogonal to $\xi$, we get
\begin{equation}
|\Ric|_{g}^2\geq (2n-1)^2|{\Hess}(f)|_{g}^2-\frac{(2n-1)^2}{2n+1}({\Delta}(f))^2+\frac{4n(2n-1)(a-1)}{a}\Delta(f)+\end{equation}$$+\frac{4n^2(2n+1)(a^2-1)}{a^2}.$$

Moreover, if $f$ is a $\Delta$-harmonic function, then
$$|\Ric|_{g}^2\geq (2n-1)^2|{\Hess}(f)|_{g}^2+\frac{4n^2(2n+1)(a^2-1)}{a^2}.$$
\end{corollary}

\begin{proposition}
If $(V=\overline{\grad}(f),\overline{\lambda})$ defines a solenoidal gradient almost Riemann soliton in a $(2n+1)$-dimensional $D$-homothetically deformed Kenmotsu manifold $(M,\overline{\phi},\overline{\xi},\overline{\eta},\overline{g})$, then
$$
|\Ric|_{g}^2\geq (2n-1)^2|{\Hess}(f)|_{g}^2+\frac{a^2-1}{a^2}[4n^2-(2n-1)^2(\xi(\xi(f)))^2].$$
\end{proposition}

\section{Almost Ricci solitons}

On a $(2n+1)$-dimensional smooth manifold $M$, a Riemannian metric $g$ and a non-vanishing vector field $V$ is said to define
 \textit{a Ricci soliton} \cite{ham} if there exists a real constant $\lambda$ such that
\begin{equation}\label{1}
\frac{1}{2}\pounds _{V}g+\Ric=\lambda g,
\end{equation}
where $\pounds _{V}$ denotes the Lie derivative operator in the direction of the vector field $V$, $\Ric$ and $\scal$ are the Ricci and the scalar curvature of $g$, respectively.
If $\lambda$ is a smooth function on $M$, we call $(V,\lambda)$ an \textit{almost Ricci soliton}. Moreover, if $V$ is a gradient vector field, we call $(V,\lambda)$ a \textit{gradient almost Ricci soliton}. A Ricci soliton defined by $(V,\lambda)$ is said to be \textit{shrinking}, \textit{steady} or \textit{expanding} according as $\lambda$ is positive, zero or negative, respectively.

Tracing (\ref{1}), we obtain:
\begin{equation}\label{sc}
\scal=(2n+1)\lambda-\div(V).
\end{equation}

\begin{theorem}\label{ppp}
If $(\overline{\xi},\overline{\lambda})$ defines an almost Ricci soliton in a $(2n+1)$-dimensional $D$-homothetically deformed Kenmotsu manifold $(M,\overline{\phi},\overline{\xi},\overline{\eta},\overline{g})$, then it is an expanding Ricci soliton and:
$$
\overline{\lambda}=-\frac{2n}{a^2},
$$
\begin{equation}\label{r222}
\Ric=-(2n+1)g+\eta\otimes \eta,
\end{equation}
\begin{equation}\label{scalar}
\scal=-4n(n+1).
\end{equation}
\end{theorem}
\begin{proof}
Equating (\ref{333}) and (\ref{1}) and using (\ref{l}) and (\ref{di}), we get:
\begin{equation}\label{B5}
\Ric=\left(a\overline{\lambda}-1-2n\frac{a-1}{a}\right)g+\left[a(a-1)\overline{\lambda}+1+2n\frac{a-1}{a}\right]\eta\otimes \eta.
\end{equation}

Also, by equating (\ref{s}) and (\ref{sc}) and using (\ref{di}), we obtain:
\begin{equation}\label{C5}
\scal=(2n+1)a\overline{\lambda}-2n-2n(2n+1)\frac{a-1}{a}.
\end{equation}

Now, tracing (\ref{B5}) and considering (\ref{C5}), we get:
$$\overline{\lambda}=-2n\frac{1}{a^2}$$
which, by replacing it in (\ref{B5}) and (\ref{C5}), gives (\ref{r222}) and (\ref{scalar}).
\end{proof}

\begin{remark}
Under the hypotheses of Theorem \ref{ppp}, we get $$|\Ric|^2=2n(4n^2+6n+3).$$
\end{remark}

\begin{proposition}
If $(\overline{\xi},\overline{\lambda})$ defines an almost Ricci soliton in $(M,\overline{\phi},\overline{\xi},\overline{\eta},\overline{g})$, then $(\xi,\lambda)$ defines an almost Ricci soliton in $(M,\phi,\xi,\eta,g)$ if and only if $\lambda=-2n$.
\end{proposition}

\begin{theorem}\label{h}
If $(V,\overline{\lambda})$ defines an almost Ricci soliton of solenoidal type in a $(2n+1)$-dimensional $D$-homothetically deformed Kenmotsu manifold $(M,\overline{\phi},\overline{\xi},\overline{\eta},\overline{g})$, then:
$$
\overline{\lambda}=\xi(\eta(V))-\frac{2n}{a^2},
$$
\begin{equation}\label{r2222}
\Ric=\left[a\xi(\eta(V))-2n\right]g(X,Y)+a(a-1)\xi(\eta(V))\eta(X)\eta(Y)-
\end{equation}
$$-\frac{a}{2}[g(\nabla_XV,Y)+g(\nabla_YV,X)]-
\frac{a(a-1)}{2}\{\eta(X)[\eta(\nabla_YV)+g(Y,V)]+$$$$+\eta(Y)[\eta(\nabla_XV)+g(X,V)]-2\eta(X)\eta(Y)\eta(V)\},
$$
\begin{equation}\label{s25}
\scal=(2n+1)[a\xi(\eta(V))-2n].
\end{equation}
\end{theorem}
\begin{proof}
Equating (\ref{333}) and (\ref{1}) and using $\div(V)=0$, we get:
\begin{equation}\label{B12}
\Ric(X,Y)=\left(a\overline{\lambda}-2n\frac{a-1}{a}\right)g(X,Y)+\left[a(a-1)\overline{\lambda}+2n\frac{a-1}{a}\right]\eta(X)\eta(Y)-\end{equation}
$$-\frac{a(a-1)}{2}\{\eta(X)[\eta(\nabla_YV)+g(Y,V)]+
\eta(Y)[\eta(\nabla_XV)+g(X,V)]-$$
$$-2\eta(X)\eta(Y)\eta(V)\}-\frac{a}{2}[g(\nabla_XV,Y)+g(\nabla_YV,X)].$$

Also, by equating (\ref{s}) and (\ref{sc}), we obtain:
\begin{equation}\label{C12}
\scal=(2n+1)a\overline{\lambda}-2n(2n+1)\frac{a-1}{a}.
\end{equation}

Now, tracing (\ref{B12}) and considering (\ref{C12}), we get:
$$\overline{\lambda}=\xi(\eta(V))-2n\frac{1}{a^2}$$
which, by replacing it in (\ref{B12}) and (\ref{C12}), gives (\ref{r2222}) and (\ref{s25}).
\end{proof}

\begin{corollary}
Under the hypotheses of Theorem \ref{h},
if $V$ is $g$-orthogonal to $\xi$, then $(M,g)$ is of negative constant scalar curvature and $(V,\overline{\lambda})$ is an expanding Ricci soliton.
\end{corollary}

\begin{theorem}\label{lll}
If $(V=\overline{\grad}(f),\overline{\lambda})$ defines a gradient almost Ricci soliton in a $(2n+1)$-dimensional $D$-homothetically deformed Kenmotsu manifold $(M,\overline{\phi},\overline{\xi},\overline{\eta},\overline{g})$, then:
\begin{equation}\label{lamx}
\overline{\lambda}=\frac{1}{a^2}\Hess(f)(\xi,\xi)-\frac{2n}{a^2}.
\end{equation}

\end{theorem}
\begin{proof}
If $V=\overline{\grad}(f)$, then $\frac{1}{2}\pounds _{V}\overline{g}=\overline{\Hess}(f)$ and equation (\ref{1}) becomes
\begin{equation}\label{rrx}
\overline{\Hess}(f)+\overline{\Ric}=\overline{\lambda}\overline{g}.
\end{equation}

Using (\ref{333}) and (\ref{he}) in (\ref{rrx}) we get
\begin{equation}\label{ppx}
\Hess(f)+\Ric+\frac{a-1}{a}[2n-\xi(f)](g-\eta\otimes \eta)=\overline{\lambda}[ag+a(a-1)\eta\otimes \eta]\end{equation}
and tracing (\ref{ppx}) w.r.t. $g$, we obtain
\begin{equation}\label{oox}
\scal=a(2n+a)\overline{\lambda}-\Delta(f)+\frac{2n(a-1)}{a}\xi(f)-\frac{4n^2(a-1)}{a}.\end{equation}

Now tracing (\ref{rrx}) w.r.t. $\overline{g}$ we get
$$
\overline{\Delta}(f)+\overline{\scal}=(2n+1)\overline{\lambda}
$$
and using (\ref{s}) and (\ref{la}), we obtain
\begin{equation}\label{ooox}
\scal=a(2n+1)\overline{\lambda}-\Delta(f)+\frac{2n(a-1)}{a}\xi(f)+\frac{a-1}{a}\xi(\xi(f))-\frac{2n(2n+1)(a-1)}{a}.\end{equation}

Now, equating (\ref{oox}) and (\ref{ooox}), taking into account that $\xi(\xi(f))=\Hess(f)(\xi,\xi)$, we get (\ref{lamx}).

\end{proof}

\begin{example}
Consider the $3$-dimensional Kenmotsu manifold $(M, \phi,\xi,\eta,g)$, where $M=\{(x,y,z)\in\mathbb{R}^3, z>1\}$, with $(x,y,z)$ the standard coordinates in $\mathbb{R}^3$, and
$$\phi:= dx\otimes \frac{\partial}{\partial y}-dy\otimes \frac{\partial}{\partial x}, \ \ \xi:=\frac{\partial}{\partial z}, \ \ \eta:= dz, \ \ g:=e^{2z}(dx\otimes dx+dy\otimes dy)+dz\otimes dz.$$

Then the pair $(V=e^z\frac{\partial}{\partial z}, \lambda=e^z-2)$ defines a shrinking gradient almost Ricci soliton \cite{bl}, with
$V=\grad(f)$, for $f(x,y,z)=e^z$.

If $(\overline{V}=\overline{\grad}(f),\overline{\lambda})$ defines a gradient almost Ricci soliton in the $D$-homothetically deformed Kenmotsu manifold $(M,\overline{\phi},\overline{\xi},\overline{\eta},\overline{g})$, then
$$\overline{V}= \frac{1}{a^2}V, \ \ \overline{\lambda}=\frac{1}{a^2}\lambda.$$
\end{example}

\begin{corollary}
Under the hypotheses of Theorem \ref{lll}, if $V$ is $\overline{g}$-orthogonal to $\xi$, the Ricci soliton is expanding and we get:
\begin{equation}\label{as1}
\overline{\lambda}=-\frac{2n}{a^2},
\end{equation}
\begin{equation}\label{ds1}
\scal=-\Delta(f)-2n(2n+1).
\end{equation}
\end{corollary}

Considering (\ref{rrx}) and computing the Hilbert-Schmidt norms, if $(V=\overline{\grad}(f),\overline{\lambda})$ defines a gradient almost Ricci soliton in a $(2n+1)$-dimensional $D$-homothetically deformed Kenmotsu manifold $(M,\overline{\phi},\overline{\xi},\overline{\eta},\overline{g})$, then
$$
|\overline{\Hess}(f)|_{\overline{g}}^2=|\overline{\Ric}|_{\overline{g}}^2-(2n+1)\overline{\lambda}^2+2\overline{\Delta}(f)\overline{\lambda}
$$
and imposing the existence condition on $\overline{\lambda}$, we get, in this case, the same estimation (\cite{blcr}, \cite{cr}):

\begin{theorem}
If $(V=\overline{\grad}(f),\overline{\lambda})$ defines a gradient almost Ricci soliton in a $(2n+1)$-dimensional $D$-homothetically deformed Kenmotsu manifold $(M,\overline{\phi},\overline{\xi},\overline{\eta},\overline{g})$, then
\begin{equation}\label{m}
|\overline{\Ric}|_{\overline{g}}^2\geq |\overline{\Hess}(f)|_{\overline{g}}^2-\frac{1}{2n+1}(\overline{\Delta}(f))^2.
\end{equation}
\end{theorem}

\begin{corollary}
If $(V=\overline{\grad}(f),\overline{\lambda})$ defines a solenoidal gradient almost Ricci soliton in a $(2n+1)$-dimensional $D$-homothetically deformed Kenmotsu ma\-ni\-fold $(M,\overline{\phi},\overline{\xi},\overline{\eta},\overline{g})$, then
\begin{equation}
|\overline{\Ric}|_{\overline{g}}^2\geq |\overline{\Hess}(f)|_{\overline{g}}^2.
\end{equation}
\end{corollary}

Replacing now $\overline{\Ric}$, $\overline{\Hess}(f)$ and $\overline{\Delta}$ from (\ref{333}), (\ref{he}) and (\ref{la}) in (\ref{m}), we obtain:
\begin{proposition}\label{hhah1}
If $(V=\overline{\grad}(f),\overline{\lambda})$ defines a gradient almost Ricci soliton in a $(2n+1)$-dimensional $D$-homothetically deformed Kenmotsu manifold $(M,\overline{\phi},\overline{\xi},\overline{\eta},\overline{g})$, then
\begin{equation}
|\Ric|_{g}^2\geq |{\Hess}(f)|_{g}^2
-4n\frac{a-1}{a}\scal-4n^2(2n+1)\left(\frac{a-1}{a}\right)^2-\frac{1}{2n+1}({\Delta}(f))^2-\end{equation}
$$-\frac{2}{2n+1}\frac{a-1}{a}[\xi(f)-\xi(\xi(f))]\Delta(f)+\frac{2(2n+a)(a-1)}{(2n+1)a^2}\xi(f)\cdot \xi(\xi(f))+$$$$
+\frac{2n}{2n+1}\left(\frac{a-1}{a}\right)^2(\xi(f))^2-\frac{2(n+na+a)(a-1)}{(2n+1)a^2}(\xi(\xi(f)))^2.$$
\end{proposition}

\begin{corollary}
Under the hypotheses of Proposition \ref{hhah1}, if $V$ is $\overline{g}$-orthogonal to $\xi$, we get
\begin{equation}
|\Ric|_{g}^2\geq |{\Hess}(f)|_{g}^2-\frac{1}{2n+1}({\Delta}(f))^2+\frac{4n(a-1)}{a}{\Delta}(f)+\frac{4n^2(2n+1)(a^2-1)}{a^2}.
\end{equation}

Moreover, if $f$ is a $\Delta$-harmonic function, then
$$|\Ric|_{g}^2\geq |{\Hess}(f)|_{g}^2+\frac{4n^2(2n+1)(a^2-1)}{a^2}.$$
\end{corollary}

\begin{proposition}
If $(V=\overline{\grad}(f),\overline{\lambda})$ defines a solenoidal gradient almost Ricci soliton in a $(2n+1)$-dimensional $D$-homothetically deformed Kenmotsu manifold $(M,\overline{\phi},\overline{\xi},\overline{\eta},\overline{g})$, then
$$
|\Ric|_{g}^2\geq |{\Hess}(f)|_{g}^2+\frac{a^2-1}{a^2}[4n^2-(\xi(\xi(f)))^2].$$
\end{proposition}

\small{

\bigskip

\textit{Adara M. Blaga}

\textit{Department of Mathematics}

\textit{West University of Timi\c{s}oara}

\textit{Bld. V. P\^{a}rvan nr. 4, 300223, Timi\c{s}oara, Rom\^{a}nia}

\textit{adarablaga@yahoo.com}

}
\end{document}